\documentclass[12pt,notitlepage]{amsart}
\usepackage{latexsym,amsfonts,amssymb,amsmath,amsthm}
\usepackage{graphicx}
\usepackage{color}
\usepackage{multirow}
\usepackage[normalem]{ulem}
\usepackage{datetime2}

\newtheorem{mytheo}{Theorem}[section]

\newtheorem{myprop}[mytheo]{Proposition}

\newtheorem{mycoro}[mytheo]{Corollary}
\newtheorem{mylemma}[mytheo]{Lemma}
\newtheorem{mydefi}[mytheo]{Definition}

\newcommand{\mymod}{\ensuremath{\negthickspace \negmedspace \pmod}}

\usepackage[inner=1.0in,outer=1.0in,bottom=1.0in, top=1.0in]{geometry}

\begin{document}

\title{Large Values of Newform Dedekind Sums}
\author{Georgia Corbett}
\address{Department of Mathematics \\ Bucknell University \\ Olin Science Building \\ Lewisburg, PA 17837 \\ USA}
\email{gsfc001@bucknell.edu}

\author{Matthew P. Young}
 
 \address{Department of Mathematics \\
 	  Texas A\&M University \\
 	  College Station \\
 	  TX 77843-3368 \\
 		U.S.A.}		
 \email{myoung@math.tamu.edu}
\thanks{This material is based upon work supported by the National Science Foundation under agreement No. DMS-2302210 (M.Y.).  Any opinions, findings and conclusions or recommendations expressed in this material are those of the authors and do not necessarily reflect the views of the National Science Foundation.
 }

\begin{abstract}
     We study a generalized Dedekind sum $S_{\chi_1,\chi_2}(a,c)$ attached to newform Eisenstein series $E_{\chi_1,\chi_2}(z,s)$. 
     Our work shows the Dedekind sum is rarely substantially larger than $\log^3 c$.  The method of proof first relates the size of the Dedekind sum to continued fractions.  A result of Hensley from 1991 then controls the average size of the maximal partial quotient in the continued fraction expansion of $a/c$. 
     
     We complement this result by 
computing approximate values of the Dedekind sum in some special cases, which in particular     
     produces examples of large values of the Dedekind sum.  
\end{abstract}
\maketitle

\section{Introduction}
%\subsection{Background and prior work}
\subsection{Background}
%The classical Dedekind sum was first introduced to describe the transformation formula of the Dedekind eta function. Other applications of classical Dedekind sums appear outside of number theory such as topology, combinatorial geometry, and mathematical physics.
%The Dedekind eta function is heavily studied as it relates to several notable functions such as the Euler function, Jacobi theta function, and partition function. 
%Dedekind sums themselves arise in various contexts such as topology, quadratic reciprocity, and modular forms. 
%%%%%%%%%%%%%%%%%%%%%%%%%%%%
% \subsection{}
%\subsection{Definitions}
Let $h$ and $k$ be coprime integers with $k > 0$. The classical Dedekind sum is defined as $$s(h,k)=\sum_{n\hspace{-.5em}\pmod{k}} B_1\left(\frac{n}{c}\right)B_1\left(\frac{hn}{k}\right),$$ 
where $B_1$ denotes the first Bernoulli function defined by $$B_1(x) =
        \begin{cases}
            x-\lfloor x\rfloor -\frac{1}{2}, & \text{if } x\in \mathbb{R}\backslash\mathbb{Z}\\
            0, & \text{if } x\in \mathbb{Z}.
        \end{cases}$$        
In this paper, we focus on a generalization of the classical Dedekind sum associated to the newform Eisenstein series introduced by Stucker, Vennos, and Young in \cite{Stucker2020}. 
\begin{mydefi}Let $\chi_1,\chi_2$ be primitive nontrivial Dirichlet characters modulo $q_1,q_2$ respectively with  $\chi_1 \chi_2(-1)=1$. Let $\gamma=  
(\begin{smallmatrix}
        a & b\\
        c & d
    \end{smallmatrix}) \in \Gamma_0(q_1q_2)$ with $c\geq 1$. Then the newform Dedekind sum is defined by
\begin{equation}\label{newform}
    S_{\chi_1, \chi_2}(\gamma)= \sum_{j\hspace{-.5em}\pmod{c}}\sum_{n \hspace{-.5em}\pmod{q_1}}\overline{\chi_2}(j) \hspace{0.1cm}\overline{\chi_1}(n) B_1\left(\frac{j}{c}\right)B_1\left(\frac{n}{q_1}+\frac{aj}{c}\right).
\end{equation}
\end{mydefi}

The value distribution of the classical Dedekind sum has been studied extensively; for instance, see  \cite{Hickerson1977, Myerson1988, CFKS} for some results in this vein. In contrast, many basic questions about the distribution of values of the newform Dedekind sums remain open. 

Dillon and Gaston \cite{DG2020} proved an approximate 
formula for its second moment:
\begin{equation}\label{DG}
    \sum_{\substack{a\hspace{-.5em}\pmod{c} \\(a,c)=1}} |S_{\chi_1,\chi_2}(a,c)|^2 = q_1c^{2 +o(1)}.
    \end{equation}  
 This result shows the newform Dedekind sum has square-root cancellation on average, noting that the trivial bound applied to \eqref{newform} yields
\begin{equation}\label{trivial bound}
    |S_{\chi_1,\chi_2}(a,c)|\leq q_1c.
\end{equation}
In this paper, we are most interested in investigating the behavior as a function of $c$, so we largely consider $q_1,q_2$ as fixed. 
A natural question is if \eqref{trivial bound} can be improved substantially.  
One might guess that $S_{\chi_1, \chi_2}(a,c) \ll c^{1/2+o(1)}$, based on the square-root heuristic consistent with \eqref{DG}.
Proposition 3.8 of \cite{addtwists} indeed gives a bound of this quality for a cuspidal variant of the newform Dedekind sum.
Perhaps
surprisingly, no improvement of \eqref{trivial bound} is possible as De Leon and McCormick \cite{DM2023} showed that $S_{\chi_1,\chi_2}(a,c)$ can have size proportional to $c$:  
 \begin{mytheo}[De Leon, McCormick]\label{DW result}
      Let $\chi$ be the Legendre symbol mod $p$, and $k, \ell \in \mathbb{Z}$ with $k\geq 1$.  Then
    \begin{equation}
        S_{\chi,\chi}(1+\ell kp,kp^2)= \chi(-\ell ) \frac{k(p^2-1)}{12}.
    \end{equation}
 \end{mytheo}

\subsection{Results}
Our first main result controls the frequency at which the newform Dedekind sum may take on large values.
 For $\chi_1$ and $\chi_2$ fixed, let
\begin{equation}
\label{eq:FalphaCdef}
    F( \alpha, C):= \#\{(a,c): 1\leq a<c\leq C, \thinspace (a,c) = 1, \thinspace q_1 q_2 | c,
\thinspace     
    |S_{\chi_1,\chi_2}(a,c)|> \alpha \log^3 C\}.
\end{equation}

\begin{mytheo}\label{thm1}
For  $C$ large and $\alpha>1$, then 
\begin{equation}
\label{eq:Fbound}
    F(\alpha, C)\ll_{\chi_1, \chi_2} \frac{C^2}{\alpha} + C^2 \frac{\log \log C}{\log C}.
\end{equation}
\end{mytheo}
Theorem \ref{thm1} shows that the Dedekind sum is rarely much larger than $\log^3 C$. 
We note that for large values of $\alpha$, say $\alpha > q_1 C$, then this upper bound is worse than the trivial bound \eqref{trivial bound}.  Also, for comparison we note that a simple application of \eqref{DG} implies
$F(\alpha, C) \ll \alpha^{-2} C^{3+o(1)}$, which is stronger than \eqref{eq:Fbound} only for $\alpha \gg C^{1/2+o(1)}$.

As a key intermediate step in the proof of Theorem \ref{thm1}, we have the following result.
\begin{mytheo}
\label{bound of Dedekind sum with partial fractions}
Let $c'=\frac{c}{q_2}$ and $\gcd(a,c)=1$. Let $D(a,c')$ be the largest partial quotient of the continued fraction expansion of $\frac{a}{c'}$.  There exists a constant $Q = Q(\chi_1, \chi_2)$ such that
$$|S_{\chi_1,\chi_2}(a,c)|\leq Q \cdot D(a,c')  \log^2 c'.$$
\end{mytheo}
Theorem \ref{bound of Dedekind sum with partial fractions} relates newform Dedekind sums to the subject of continued fractions which has a rich literature. In \cite{Hensley1991}, Hensley proved a formula for the frequency of values for which $D(a,c')\ll \log c'$ (see Theorem \ref{Hensley} below for the precise statement). Hensley's result is key in our proof of Theorem \ref{thm1}. 

Theorem \ref{bound of Dedekind sum with partial fractions} shows that the Dedekind sum may take a large value only when $a/c'$ has a large partial quotient.  In light of this apparent obstruction, it is natural to study the Dedekind sum at such pairs $(a,c)$.  The most obvious choices to consider are of the form $a=1$, $1+c'$, $2+c'$, etc.
Note that in the notation from Theorem \ref{DW result}, $\frac{a}{c'} = \frac{1 + \ell kp}{kp} = \ell + \frac{1}{kp}$, so $D(a,c') = k p$ in this case, which is consistent with Theorem \ref{bound of Dedekind sum with partial fractions} since the implied constant therein is allowed to depend on $p$.  As a partial generalization of Theorem \ref{DW result}, we have the following result which produces large values of the Dedekind sum for more general pairs of characters (with a light technical restriction).
\begin{mytheo}
\label{large value}
    Suppose $a=1+nc'$ and $d = 1 - mc'$, with $ad \equiv 1 \pmod{c}$.
    Then 
    \begin{equation*}
        S_{\chi_1,\chi_2}(a,c) = \beta(\chi_1, \chi_2, m,n) c' + O(\log c'),
    \end{equation*}
    where
    \begin{equation}
    \beta(\chi_1, \chi_2, m,n) = 
    \frac{\tau(\overline{\chi_1}) \tau(\overline{\chi_2})}{4 \pi^2 i} L(2, \chi_1 \chi_2)
\Big((1+i) \chi_2(n) - (1-i) \chi_2(m)  \overline{\chi_2}(d) \Big).
    \end{equation}
\end{mytheo}
For ``typical" values of $n$, we should expect that $\beta \neq 0$, in which case $|S_{\chi_1, \chi_2}(a,c)| \gg c'$.  
For example, suppose that $q_1 = q_2=q$.  In this case, $(c')^2 \equiv 0 \pmod{c}$, since 
$(c')^2 = c \frac{c}{q^2}$, and $q^2|c$.
Therefore, $m \equiv n \pmod{q}$, and $\chi_2(d) = 1$.  Hence, in this case,
\begin{equation}
\beta(\chi_1, \chi_2, n, n) = \chi_2(n)
\frac{\tau(\overline{\chi_1}) \tau(\overline{\chi_2})}{2 \pi^2} L(2, \chi_1 \chi_2).
\end{equation}
Suppose in addition that $\chi_2 = \overline{\chi_1}$;  then
$\tau(\overline{\chi_1}) \tau(\overline{\chi_2}) = q \chi_2(-1)$, and
$L(2, \chi_1 \chi_2) = \zeta(2) \prod_{p|q} (1-p^{-2})$.  Therefore
\begin{equation}
\beta(\overline{\chi_2}, \chi_2, n, n)
= \chi_2(-n) \frac{q}{12} \prod_{p|q} (1-p^{-2}).
\end{equation}
Hence letting $r = \prod_{p|q} p^2$ we  deduce
\begin{equation}
S_{\overline{\chi_2}, \chi_2}(1+n \frac{c}{q},c) = 
\chi_2(-n) \frac{c}{12 r} \prod_{p|q} (p^2 - 1) + O(\log c).
\end{equation}
If $p>3$, then $p^2 - 1 \equiv 0 \pmod{12}$, so this constant is an integer multiple of $\chi_2(n)$ provided $q$ has a prime factor coprime to $6$.

Also, as a safety check, when $q=p$ prime and $\chi_1 = \chi_2$ is the Legendre symbol, we recover a result consistent with Theorem \ref{DW result} (which is stronger of course since it has no error term).

Another comment is that in case $m=n=0$ then $\beta(\chi_1, \chi_2, 0,0) = 0$.  This is consistent with 
\cite[Corollary 2.7]{NRY} which says $S_{\chi_1, \chi_2}(1,c) = 0$.

\subsection{Connections and future directions}
Up to a normalizing factor, the newform Dedekind sum can be realized as the modular symbol associated to a weight $2$ Eisenstein series.  Specifically, let
\begin{equation}
E_{\chi_1, \chi_2}(z) = 2 \sum_{n=1}^{\infty} n^{1/2} \lambda_{\chi_1,\chi_2}(n) e(nz),
\quad \lambda_{\chi_1, \chi_2}(n) = 
\sum_{ab=n} \chi_1(a) \overline{\chi_2}(b) \Big(\frac{b}{a}\Big)^{1/2},
\end{equation}
which is a holomorphic weight $2$ Eisenstein series on $\Gamma_0(q_1 q_2)$ with central character $\chi_1 \overline{\chi_2}$.
For $\gamma \in \Gamma_0(q_1 q_2)$, 
\begin{equation}
\int_{i \infty}^{\gamma(\infty)} E_{\chi_1, \chi_2}(z) dz = \tau(\overline{\chi_1}) S_{\chi_1, \chi_2}(\gamma),
\end{equation}
which indeed is a modular symbol.
See \cite[Section 5]{Stucker2020} for more details.

There are a large number of works on the distribution of values of modular symbols.
For instance, see 
\cite{PetridisRisager, addtwists, Nordentoft, BettinDrappeau, Cowan, lee2023dynamics} for some recent advances in the subject, which largely (if not entirely) restrict attention to modular symbols for cusp forms.  The case of Eisenstein series considered in this article has some familiar extra difficulties due to the lack of vanishing at all cusps.
It could be valuable to investigate possible extensions of these methods and results to the case of Eisenstein series.

\section{Continued Fractions}
We introduce a discussion on properties of continued fractions as they are a crucial part of the proofs of our results.
Let $\frac{c}{d}$ be a rational number with $d \geq 1$ and $\gcd(c,d)=1$. The simple finite continued fraction expansion of $\frac{c}{d}$ is 
$$\frac{c}{d}=
a_0+\cfrac{1}{a_1+\cfrac{1}{a_2+\cfrac{1}{\ddots+\cfrac{1}{a_n }
  }}}$$
  where 
$a_0 \in \mathbb{Z}$,  
  $a_i \in \mathbb{Z}^+$ for $i \geq 1$, and $n\in \mathbb{Z}^{+}$. 
It is traditional to write this more compactly as $[a_0;a_1, \dots, a_n]$.   
  The partial quotients are the $a_i \in \mathbb{Z}^+$ where $1\leq i \leq n$.
Each rational number has precisely two continued fraction expansions, via 
\begin{equation}
\label{eq:continuedfractionmultiplicity}
[a_0;a_1, \dots, a_{k-1}, a_k, 1] = [a_0;a_1, \dots, a_{k-1}, a_k + 1].
\end{equation}
Note that exactly one of these two expansions has an even number of terms.
  The maximal partial quotient is defined by $D(a,b)= \max\{a_i: 1\leq i\leq n\}$, where in light of 
the potential ambiguity arising from  
  \eqref{eq:continuedfractionmultiplicity}  we assume that $a_n \geq 2$. 
  Following \cite{Hensley1991}, for $C\in \mathbb{Z}^+$ and $\alpha>0$, define $$ \Phi(\alpha,C)=\#\{(a,c): 1<a<c\leq C, \gcd(a,c)=1, D(a,c)\leq \alpha \log C\}.$$  
  \begin{mytheo}[Hensley]\label{Hensley}
  We have
      \begin{equation}
      \Phi(\alpha,C)=\frac{3}{\pi^2}C^2 \exp\left(\frac{-12}{\alpha \pi^2}\right) \left(1+ O\left (\left ( \alpha^{-2}+1\right)\exp\Big(\frac{24}{\alpha \pi^2}\Big) \frac{\log \log C}{\log C} \right )\right),
       \end{equation} 
       uniformly in $\alpha > \frac{4}{ \log\log C}$ as $C \rightarrow \infty$. 
  \end{mytheo}
%\red{Check: do we only need this for $\alpha > 1$?  In that case it would simplify.}

Hensley's result gives the frequency of small partial quotients in the continued fraction expansion of rationals in the set $\{(a,c): 1<a<c\leq C, \gcd(a,c)=1\}$. We can re-frame this result to bound the frequency of large partial quotients instead. Let 
\begin{equation}
\label{eq:GalphaCdef}
 G(\alpha,C)=\#\{(a,c): 1<a<c\leq C, \gcd(a,c)=1, D(a,c)>\alpha \log C\}.
 \end{equation}
\begin{mycoro}\label{large partial quotient}
   We have
      \begin{equation}G(\alpha,C)\ll  \frac{C^2}{\alpha} + C^2 \frac{\log \log C}{\log C}
       \end{equation} 
       uniformly in $\alpha > 1$ as $C \rightarrow \infty$. 
\end{mycoro}

\begin{proof}
    Note that $G(\alpha,C)+\Phi(\alpha,C)=
\sum_{c \leq C} \varphi(c) =     
    \frac{3}{\pi^2}C^2+O( C \log C)$.
    Then 
    \begin{align*}
        G(\alpha,C)&= \frac{3}{\pi^2}C^2-\Phi(\alpha,C)+O(C \log{C})\\
        &= \frac{3}{\pi^2}C^2\left (1-\exp\left(\frac{-12}{\alpha \pi^2}\right)\right)+O\Big(C^2 \frac{\log \log C}{\log C} \Big).
    \end{align*}
The proof is completed with a Taylor expansion.
\end{proof}

For later purposes, we also need the following result on continued fractions.
\begin{mylemma}
\label{lemma:reversal}
Suppose that $(\begin{smallmatrix} a & b \\ c & d \end{smallmatrix}) \in SL_2(\mathbb{Z})$ with $0 < a < c$ and $0 < d < c$.  Say the continued fraction expansion of $a/c$ takes the form $[0;a_1, a_2, \dots, a_n]$ with $n$ odd.  Then the continued fraction expansion of $d/c$ equals $[0;a_n, \dots, a_1]$.
In particular, $D(a,c) = D(d,c)+ \delta$ for some $|\delta| \leq 1$.
\end{mylemma}
\begin{proof}
We begin by drawing some material from \cite{Poorten}.  The continued fraction expansion $a/c = [a_0;a_1, \dots, a_n]$ corresponds to the matrix factorization
\begin{equation}
\label{eq:matrixfactorization}
\begin{pmatrix} a & B \\ c & D \end{pmatrix}
=
\begin{pmatrix} a_0 & 1 \\ 1 & 0 \end{pmatrix}
\begin{pmatrix} a_1 & 1 \\ 1 & 0 \end{pmatrix}
\dots
\begin{pmatrix} a_n & 1 \\ 1 & 0 \end{pmatrix} \in SL_2(\mathbb{Z}),
\end{equation}
with $B/D = [a_0;a_1, \dots, a_{n-1}]$.
To be precise, one of the two continued fraction expansions of $a/c$ gives rise to this factorization, and the ambiguity is resolved by comparing determinants of both sides, leading to the restriction that $n$ is odd.  We first claim that $B = b$ and $D=d$.  To see this, note that $0 < D < c$ since $B/D$ is a prior convergent of the fraction $a/c$. 
Comparing determinants, we have $d \equiv D \pmod{c}$.
Since $0 < d, D < c$ then $D=d$, and hence $B=b$ as well.

Van der Poorten argues that by taking the transpose of both sides of \eqref{eq:matrixfactorization} we have
\begin{equation}
[a_n;a_{n-1}, \dots, a_0] = \frac{a}{b},
\qquad \text{and}
\qquad
[a_n;a_{n-1}, \dots, a_1] = \frac{c}{d}.
\end{equation}
Taking the reciprocal of the latter equation implies 
$\frac{d}{c} = [0;a_n, a_{n-1}, \dots, a_1]$, as desired.

For the final sentence, we need to observe that $D(a,c)$ was defined using the representation of the continued fraction with final element $a_n \geq 2$, which may differ from the stated form (with $n$ odd).  The same feature can occur in the reversed continued fraction.  Some thought shows that $D(a,c)$ and $D(d,c)$ may differ by at most $1$.
\end{proof}
\begin{mycoro}
\label{coro:digitsymmetry}
Suppose $\gamma=(\begin{smallmatrix}
    a & b\\
    c & d
 \end{smallmatrix}) \in \Gamma_0(q_1q_2)$ and set $c' = c/q_2$.  
 Then $D(a,c') = D(d,c') + \delta$, where $|\delta| \leq 1$.  
\end{mycoro}
\begin{proof}
Note that
$\gamma' = (\begin{smallmatrix} a & q_2 b \\ c' & d \end{smallmatrix}) \in SL_2(\mathbb{Z})$.
We can multiply $\gamma'$ on the left or right by matrices of the form $T^j = (\begin{smallmatrix} 1 & j \\ 0 & 1 \end{smallmatrix})$ without changing the largest partial quotient of $a/c'$ or $d/c'$.  By appropriate choices we obtain $\gamma' = T^j (\begin{smallmatrix} a' & * \\ c' & d' \end{smallmatrix}) T^k$, with
$a' \equiv a \pmod{c'}$, $0 < a' < c'$, $d' \equiv d \pmod{c'}$, and $0 < d' < c'$.  Then $D(a,c') = D(a', c') = D(d', c') +\delta = D(d, c') + \delta$, where the middle equation uses Lemma \ref{lemma:reversal}.
\end{proof}

%-----------------------------------------------------
\section{Proof of Theorem \ref{bound of Dedekind sum with partial fractions} }
\subsection{Alternative form of $S_{\chi_1,\chi_2}(a,c)$}
We introduce an alternative formulation of \eqref{newform} which will be used for the remainder of the paper. Define for $z \in \mathbb{H}$ 
\begin{equation} \label{Ffunction}
    f_{\chi_1, \chi_2}(z)=\sum_{l=1}^{\infty}\sum_{k=1}^{\infty} \frac{\chi_1(l)\overline{\chi_2}(k)}{l} e(klz),
\end{equation}
 where $e(klz)=e^{2 \pi i klz}$.
To help orient the reader, we mention that $f$ is an Eichler integral of the holomorphic weight $2$ Eisenstein series attached to the pair of characters $\chi_1$, $\chi_2$.  It can alternatively be realized as the holomorphic part of the constant term at $s=1$ of $E_{\chi_1, \chi_2}(z,s)$.
These facts will not be used directly anywhere in this article.

 Then define  
 \begin{equation} 
 \label{Phi}
    \phi_{\chi_1, \chi_2}(\gamma,z)= f_{\chi_1, \chi_2}(\gamma z)-\psi(\gamma) f_{\chi_1, \chi_2}(z),
\end{equation}
where $\psi=\chi_1\overline{\chi_2}$. Lemma 2.1 in \cite{Stucker2020} gives a straightforward argument showing that $\phi$ is independent of $z$. 
%We denote $\phi_{\chi_1, \chi_2}(\gamma)$ for the remainder of the paper. 
In \cite{Stucker2020}, the authors introduced an alternative formulation of \eqref{newform}:
\begin{mydefi}[SVY2020]
    Let $\chi_1,\chi_2$ be primitive nontrivial Dirichlet characters modulo $q_1,q_2$ respectively with  $\chi_1 \chi_2(-1)=1$. Let $\gamma= \left(\begin{smallmatrix}
    a & b\\
    c & d
    \end{smallmatrix} \right) \in \Gamma_0(q_1q_2)$ with $c\geq 1$. Then 
    \begin{equation} \label{sum1}
    S_{\chi_1,\chi_2} (a,c)= \frac{\tau(\overline{\chi_1})}{\pi i}\phi_{\chi_1,\chi_2}(\gamma)
\end{equation}
where $\tau(\overline{\chi_1})$ is the Gauss sum defined by $\tau(\overline{\chi_1})=\sum_{n \mymod{q_1}} \overline{\chi_1}(n)e_{q_1}(n)$.
\end{mydefi}
A further discussion of the relationship between \eqref{newform} and \eqref{sum1}  is found in \cite{Stucker2020}.  In fact, \eqref{sum1} was the original definition in \cite{Stucker2020}, and \eqref{newform} was deduced only with some work.

It is well-known that $|\tau(\overline{\chi_1})|= \sqrt{q_1}$, and therefore
$    |S_{\chi_1,\chi_2}(a,c)|= \frac{\sqrt{q_1}}{\pi}|\phi_{\chi_1,\chi_2}(\gamma)|$.
Hence our investigation on large values focuses on $\phi_{\chi_1,\chi_2}$.

The idea in estimating $\phi$ comes from choosing $z \in \mathbb{H}$ so that the minimum of the imaginary parts of $z$ and $\gamma z$ is maximized.  By a little computation, this is attained when
$z = -\frac{d}{c} + \frac{i}{c}$, in which case
$\gamma z = \frac{a}{c} + \frac{i}{c}$.

%------------------------------------------------
\subsection{Simplified expression of $f_{\chi_1,\chi_2}$}
Here we derive an expression for $f_{\chi_1, \chi_2}$ by evaluating a geometric series.  
\begin{mylemma}\label{lemma simplifed}
Let $c' = \frac{c}{q_2}$, and set $\theta(c',a, l)= e\left( \frac{l(a+i)}{c'} \right )$.
We have
    \begin{equation}\label{simplified f gamma z}
        f_{\chi_1, \chi_2}\Big(\frac{a}{c} + \frac{i}{c}\Big)= \sum_{l=1}^\infty  \frac{\chi_1(l)}{l\left(1-\theta(c',a,l)\right)}\sum_{k_0=1}^{q_2} \overline{\chi_2}(k_0) e\left( \frac{k_0 l (a+i)}{c} \right ).
    \end{equation}
\end{mylemma}
\begin{proof}
Note 
\begin{align}\label{f gamma z}
    f_{\chi_1,\chi_2}\left(\frac{a}{c}+\frac{i}{c}\right) &= \sum_{l=1}^{\infty} \sum_{k=1}^{\infty} \frac{\chi_1(l)}{l} \overline{\chi_2}(k) e\left( \frac{kl(a+i)}{c} \right ).
\end{align}
Write $k = k_0 + m q_2$ where $1 \leq k_0 \leq q_2$ and $m \geq 0$.
Then
\begin{equation}
    f_{\chi_1,\chi_2}\left(\frac{a}{c}+\frac{i}{c}\right)=  \sum_{l=1}^{\infty}   \sum_{k_0 =1}^{q_2} \frac{\chi_1(l)}{l} \overline{\chi_2}(k_0) e\left( \frac{k_0 l (a+i)}{c} \right ) \sum_{m=0}^{\infty} \theta(c',a,l)^m.
\end{equation}
Notice $|\theta(c',a,l)|<1$ so the sum over $m$ evaluates in closed form, giving the result.
\end{proof}

%-------------------------------------------------
\subsection{Bounding $f_{\chi_1,\chi_2}$}
Here we produce a bound on $f$ using Lemma \ref{lemma simplifed}.
As a first step, we have the following.
\begin{mylemma}
\label{f gamma z up to c}
Let $\gamma=\left( \begin{smallmatrix}
    a & b\\
    c & d
\end{smallmatrix}\right) \in \Gamma_0(q_1q_2)$ with $c\geq 1$ and let $c' = \frac{c}{q_2}$. Then 
\begin{equation}
        f_{\chi_1, \chi_2} \left(\frac{a}{c}+\frac{i}{c}\right)= \sum_{l=1}^{c'-1}  \frac{\chi_1(l)}{l\left(1-\theta(c',a,l)\right)}\sum_{k_0=1}^{q_2} \overline{\chi_2}(k_0) e\left( \frac{k_0 l (a+i)}{c} \right ) + O_{q_2}(1).
    \end{equation}
\end{mylemma}
\begin{proof}
For this, we need to estimate the tail, i.e., the sum over $l \geq c'$ in \eqref{simplified f gamma z}.
Note that $|\theta(c', a, l)| = \exp(- 2\pi l/c') < 1/2$ for $l \geq c'$.  Hence $|1 - \theta(c',a,l)|^{-1} \leq 2$.  
In addition, we have 
$\max_{k_0} |\overline{\chi_2}(k_0) e(k_0 l a/c) 
\exp(- 2\pi k_0 l/c)| \leq \exp(- 2\pi l/c)$.
Hence the contribution of these terms into \eqref{simplified f gamma z} is bounded in absolute value by
\begin{equation}
\sum_{l > c'} \frac{2}{l} q_2 \exp(- 2 \pi l/c) 
\leq 2q_2 \int_{c'}^{\infty} \frac{1}{t} \exp\Big(- \frac{2 \pi t}{q_2 c'}\Big) dt
\leq 2q_2 \log(e q_2). 
\qedhere
\end{equation}
\end{proof}

To go further with Lemma \ref{f gamma z up to c}, we will need \cite[Lemma 3]{Korobov1992}: 
\begin{mylemma}\label{Korobov}
Let $q$ be a positive integer, $1 \leq a < q$, and $\gcd(a, q) = 1$. Then
\begin{equation}\label{Korobov1}
    \sum_{1\leq l\leq q-1}\frac{1}{\|la/q\|}\leq 2 q \log q,
\end{equation}
and
\begin{equation}\label{Korobov2}
    \sum_{1\leq l\leq q-1}\frac{1}{l\|la/q\|}\leq 18 D(a,q) \log^2 q.
\end{equation}
\end{mylemma}

Now we are equipped to prove the following bound on $f$.
 \begin{myprop}\label{bounding f gamma z with partial quotient}
Let $c' = \frac{c}{q_2}$.  Then    
     $$\left|f_{\chi_1,\chi_2}\left(\frac{a}{c}+\frac{i}{c}\right)\right|\ll_{q_1, q_2} D(a,c') \cdot \log^2 c'.$$
%where recall $D(a,c')$ is the largest partial quotient of $\frac{a}{c'}$.
 \end{myprop}

\begin{proof}
From Lemma \ref{f gamma z up to c} and the triangle inequality, we deduce
\begin{align}
    \left | f_{\chi_1,\chi_2}\left(\frac{a}{c}+\frac{i}{c}\right)\right| &\leq
    q_2 \sum_{1 \leq l \leq c'-1}   \frac{|\chi_1(l)|}{l\left |1-\theta(c',a,l)\right|}  + O(1).
\end{align}
Further, it holds that $|1-e^{i\phi}|\leq 2 | 1-r e^{i\phi}|$
%\begin{equation*}
    %|1-e^{i\phi}|\leq 2 | 1-r e^{i\phi}|
%\end{equation*} 
for all $0\leq r\leq 1$ and $\phi\in \mathbb{R}$. This can be proved using elementary calculus or trigonometry techniques and is left as an exercise for the reader.  Similarly, it is an exercise to check that $|1-e(x)|^{-1} \leq 4^{-1} \| x\|^{-1}$ for all $x \in \mathbb{R}$.  
Then \begin{equation*}
     |1-\theta(c',a,l)|^{-1} \leq 2\cdot |1-e(la/c')|^{-1} \leq 2^{-1} \|l a /c' \|^{-1}.
\end{equation*}
Substituting, this shows
\begin{equation}
\left| f_{\chi_1,\chi_2}\left(\frac{a}{c}+\frac{i}{c}\right)\right|  \leq 
    \frac{q_2}{2} \sum_{1 \leq l \leq c'-1} \frac{1}{l \|a l/c' \|} + O(1). 
\end{equation}
Applying \eqref{Korobov2} completes the proof.
\end{proof}

 \subsection{Proof of Theorem \ref{bound of Dedekind sum with partial fractions}}
 %\begin{proof}
 Proposition \ref{bounding f gamma z with partial quotient} showed
$ \left| f_{\chi_1,\chi_2}\left(\frac{a}{c}+\frac{i}{c}\right)\right| \ll D(a,c') \log^2 c'$.
By symmetry, we have 
$    \left | f_{\chi_1,\chi_2}\left(\frac{-d}{c}+\frac{i}{c}\right)\right|\ll D(d,c')\log^2 c'$.
Using Corollary \ref{coro:digitsymmetry}, \eqref{Phi} and \eqref{sum1} completes the proof.
%\end{proof}

%--------------------------------------------------
\section{Proof of Theorem \ref{thm1}}
\begin{proof}
Recall $Q$ is a constant such that $|S_{\chi_1, \chi_2}(a,c)| \leq Q \cdot D(a,c') \log^2 c'$, as in Theorem \ref{bound of Dedekind sum with partial fractions}.
Also recall the definitions \eqref{eq:FalphaCdef} and \eqref{eq:GalphaCdef}.  
Then we have
\begin{align*}
F(Q \alpha, C) &= 
\# \{ (a,c): 1 \leq a < c \leq C, \thinspace q_1 q_2 | c, \thinspace (a,c) = 1, \thinspace |S_{\chi_1, \chi_2}(a,c)| \geq Q \alpha \log^3 C \} 
\\
&\leq 
\# \{ (a,c): 1 \leq a < c \leq C, \thinspace q_1 q_2 | c, \thinspace (a,c) = 1, \thinspace D(a,c') \geq  \alpha \log C \} 
\\
& \leq 
q_2 \cdot \# \{(a, c'): 1 \leq a < c' \leq \frac{C}{q_2}, \thinspace q_1 | c', \thinspace (a,c') = 1, \thinspace D(a,c') \geq \alpha \log C \}
\\
& \leq q_2 \cdot G\Big(\alpha, \frac{C}{q_2}\Big),
\end{align*}
where in the third line we used that $a$ runs over $q_2$ intervals modulo $c'$, and in the fourth line we dropped the condition $q_1 | c'$ by overcounting.  The proof is completed using Corollary \ref{large partial quotient}, and by replacing $\alpha$ by $\alpha/Q$.
\end{proof}

%---------------------------------------------------
\section{Proof of Theorem \ref{large value}} 
Our proof rests on the following approximation.
\begin{myprop}
\label{prop:fapprox}
Let $\gamma= \left(\begin{smallmatrix}
         a & b\\
         c & d
         \end{smallmatrix}\right ) \in \Gamma_0(q_1q_2)$,
$c' = \frac{c}{q_2}$, and suppose $a = \delta + nc'$, for $\delta \in \{ \pm 1 \}$.  Then
\begin{equation}
f_{\chi_1, \chi_2}\Big(\frac{a}{c} + \frac{i}{c}\Big)
= c' \frac{1+\delta i}{4 \pi} \tau(\overline{\chi_2}) L(2, \chi_1 \chi_2) \chi_2(n) 
+ O_{q_1,q_2}(\log{c'}).
\end{equation}         
\end{myprop} 
\begin{proof}
We continue from Lemma \ref{f gamma z up to c}, and begin by substituting $a = \delta + nc'$, giving
that 
\begin{equation}
f_{\chi_1, \chi_2}\Big(\frac{a}{c} + \frac{i}{c}\Big) = 
\sum_{l=1}^{c'-1}  \frac{\chi_1(l)}{l(1- e(\frac{l (\delta+i)}{c'}))}
\sum_{k_0=1}^{q_2} \overline{\chi_2}(k_0) 
e\Big(\frac{k_0 l n}{q_2}\Big)
e\Big( \frac{k_0 l (\delta+i)}{c} \Big) + O(1).
\end{equation}
By a Taylor series, we have for $l \leq c'$ that
\begin{equation}
\Big(1 - e\Big(\frac{l (\delta+i)}{c'}\Big)\Big)^{-1} = 
\frac{1+i\delta }{4 \pi} \frac{c'}{l} \Big(1+O\Big(\frac{l}{c'}\Big)\Big).
\end{equation}
Similarly,
\begin{equation}
e\Big( \frac{k_0 l (\delta+i)}{c} \Big) = 1  + O\Big(\frac{l}{c'}\Big).
\end{equation}
Hence the inner sum over $k_0$ approximates to a Gauss sum (where recall $\chi_2$ is primitive modulo $q_2$), so we have
\begin{equation}
\sum_{k_0=1}^{q_2} 
\overline{\chi_2}(k_0) 
e\Big(\frac{k_0 l n}{q_2}\Big)
e\Big( \frac{k_0 l (\delta+i)}{c} \Big)
= \chi_2(l n) \tau(\overline{\chi_2}) + O_{q_2}(l/c').
\end{equation}
Therefore,
\begin{equation}
f_{\chi_1,\chi_2}\Big(\frac{a}{c} + \frac{i}{c}\Big) = c'  \frac{1+ \delta i}{4 \pi} \tau(\overline{\chi_2}) \chi_2(n) \sum_{l \leq c'} \frac{\chi_1(l) \chi_2(l)}{l^2} 
+ O(\log c').
\end{equation}
Extending the sum to $l > c'$ incurs an error term of size $O(1)$.  Assembling everything now completes the proof.
\end{proof}

\begin{proof}[Proof of Theorem \ref{large value}]
Suppose
$\gamma= \left(\begin{smallmatrix}
         a & b\\
         c & d
         \end{smallmatrix}\right ) \in \Gamma_0(q_1q_2)$, and
 $a=1 + nc'$.  Then $ad \equiv 1 \pmod{c}$, which by reduction modulo $c'$ implies $d = 1- mc'$ for some $m$.  Then, of course, $-d = -1 + m c'$.  By applying Proposition \ref{prop:fapprox} with $\frac{-d}{c}$, we obtain
\begin{equation}
f_{\chi_1, \chi_2}\Big(\frac{-d}{c} + \frac{i}{c}\Big)
= c' \frac{1- i}{4 \pi} \tau(\overline{\chi_2}) L(2, \chi_1 \chi_2) \chi_2(m) 
+ O_{q_1,q_2}(\log{c'}).
\end{equation}  
Hence by \eqref{Phi} we have
\begin{equation}
\phi_{\chi_1, \chi_2}(\gamma)
= c' \frac{\tau(\overline{\chi_2})}{4 \pi} L(2, \chi_1 \chi_2)
\Big((1+i) \chi_2(n) - (1-i) \chi_2(m) (\chi_1 \overline{\chi_2})(d) \Big)
+ O(\log{c'})
.
\end{equation}
Note $\chi_1(d) = 1$ since $q_1 | c'$ which implies $d \equiv 1 \pmod{q_1}$.  
Substituting this into \eqref{sum1}
implies Theorem \ref{large value}.
\end{proof}

\bibliographystyle{alpha}
\bibliography{Biblio}
\end{document}